\documentclass[a4paper,fleqn]{cas-sc}

\usepackage[numbers]{natbib}

\usepackage{amsthm}
\usepackage{amssymb}
\usepackage{amsmath}
\newtheorem{theorem}{Theorem}[section]
\newtheorem{lemma}{Lemma}[section]
\newtheorem{method}{Method}[section]
\newtheorem{exam}{Example}[section]
\newtheorem{rem}{Remark}[section]
\newtheorem{col}{Corollary}[section]

\newtheoremstyle{noparens}%
{}{}%
{\itshape}{}%
{\bfseries}{.}%
{ }%
{\thmname{#1}\thmnumber{ #2}\mdseries\thmnote{ #3}}

\usepackage{hyperref}

\def\tsc#1{\csdef{#1}{\textsc{\lowercase{#1}}\xspace}}
\tsc{WGM}
\tsc{QE}
\tsc{EP}
\tsc{PMS}
\tsc{BEC}
\tsc{DE}

\begin{document}
\let\WriteBookmarks\relax
\def\floatpagepagefraction{1}
\def\textpagefraction{.001}
\shorttitle{}
\shortauthors{Dong-Kai Li et~al.}

\title [mode = title]{A relaxation accelerated two-sweep modulus-based matrix splitting iteration method for solving linear complementarity problems}                      



\author{Dong-Kai Li}


\address{Jiangsu Key Laboratory for NSLSCS, School of Mathematical Sciences, Nanjing Normal University, Nanjing 210023, PR China}

\author{Li Wang}
\cormark[1]

\author{Yu-Ying Liu}





\cortext[cor1]{Corresponding author.\\	
	E-mail addresses: \href{mailto: wangli1@njnu.edu.cn}{wangli1@njnu.edu.cn}, \href{mailto: wlsha@163.com}{wlsha@163.com}}


\begin{abstract}
 For a linear complementarity problem, we present a relaxaiton accelerated two-sweep matrix splitting iteration method. The convergence analysis illustrates that the proposed method converges to the exact solution of the linear complementarity problem when the system matrix is an $H_+$-matrix and the convergence conditions are given. Numerical experiments show that the proposed method is more efficient than the existing ones.
\end{abstract}


\maketitle

\section{Introduction}
\label{sec:1}
Consider the linear complementarity problem, which is abbreviated as LCP $(q, A)$
\begin{equation}\label{equation}
z\geq 0, r:= Az + q\geq 0, z^{T}r=0,
\end{equation}
where $z\in \mathbb{R}^{n}$ is the unknown vector to be found, $A\in \mathbb{R}^{n\times n}$ and $q\in \mathbb{R}^n$ are given real matrix and real vector, respectively.

The LCP$(q, A)$ arises from scientific computing and engineering applications has gained numerous attention in recent years. It plays a very important role in many practical applications, e.g., the American option pricing model, the economies with institutional restrictions upon prices, the optimal stopping in Markov chain, the free boundary problems, and etc. see \cite{Berman1994, Cottle1992, Murty1988} for details.

A great deal of numerical methods have been proposed to solve LCP $(q,A)$. Among them, van
Bokhoven \cite{Bokhoven1981Piecewise} introduced a modulus iteration method by rewriting the linear complementarity problem as an implicit fixed point equation, which was called “modulus algorithms”. Kappel and Watson \cite{Kappel1986} Extend it to block module algorithm. Schäfer \cite{SCHAFER2004350} did further investigate on the modulus algorithm, then, a nonstationary extrapolated modulus algorithms was proposed  by Hadjidimos and Tzoumas\cite{Hadjidimos2009}. Dong and Jiang \cite{Dong2009} investigated a modified modulus method. On the basis of modular algorithm, Zheng and Li \cite{Zheng2016A} gave the non-modulus linear method for solving the linear complementarity problem.

Specially, Bai proposed \cite{Bai2010} the modulus-based matrix splitting iteration method based on the above modulus iteration methods. He did a series of work after this, see \cite{Bai2013A, Bai2013B, Bai2017ModulusA, Zhong2017ModulusB} for detail. In addition, Zhang \cite{Zhang2011, Zhang2015} introduced two-step modulus-based matrix splitting iteration methods, Zheng and Yin \cite{Zheng2013Accelerated} presented accelerated modulus-based matrix splitting iteration methods, Zheng and Li \cite{Hua2017A} proposed a relaxation modulus-based matrix splitting iteration method by utilizing relaxation strategies. Besides, Wu and Li \cite{Wu2016} investigated two-sweep modulus-based matrix splitting iteration method, Ren and Wang \cite{Ren2019} extended it to the general two-sweep modulus-based matrix splitting iteration method.

In this paper, we presented a relaxation accelerated two-sweep modulus-based matrix splitting iteration method based on the work in \cite{Ren2019}. We give two matrix splittings in this method, then, multiple parameter diagonal matrices are introduced and relaxation strategies are used. In addition, the convergence and parameter selection strategy of the method are discussed where the system matrix is an $H_+$-matrix. Numerical results show that the proposed method is effective and has better computational efficiency than other methods.

The outline of the paper is as follows. Some notations and lemmas are listed in \autoref{sec:2}. A relaxaiton accelerated two-sweep modulus-based matrix splitting iteration method is presented in \autoref{sec:3}. The convergence analysis is investigated in \autoref{sec:4} where the system
matrix is an $H_+$-matrix, and some numerical results are given in \autoref{sec:5}, and finally we give some conclusions in \autoref{sec:6}.

\section{Preliminaries}
\label{sec:2}
Some notations, definitions and necessary lemmas are listed which will be used in the sequel.

Let two matrices $P=(p_{ij})\in \mathbb{R}^{m\times n}$ and $Q=(q_{ij})\in \mathbb{R}^{m\times n}$, we write  $P\geq Q (P > Q)$ if $p_{ij}\geq q_{ij} (p_{ij}> q_{ij})$ holds for any $i$ and $j$. For $A = (a_{ij})\in \mathbb{R}^{m \times n}$, $\vert A \vert$, $A^{T}$ and $\rho (A)$ represent the absolute value of $A$ ($\vert A \vert = (\vert a_{ij}\vert)\in \mathbb{R}^{m \times n}$), the transpose and the spectral radius of $A$, respectively. Moreover, an comparison matrix $\langle A \rangle$ is defined by
\begin{equation*}
\langle a_{ij}\rangle = \begin{cases}
\vert a_{ij}\vert, &if\quad i = j,\\
-\vert a_{ij}\vert, &if \quad i \not = j,\\
\end{cases}
\quad i,j=1, 2, \cdots ,n.
\end{equation*}

A matrix $A\in \mathbb{R}^{n\times n}$ if all of its off-diagonal entries are nonpositive, is called a $Z$-matrix, and it's a $P$-matrix if all of its principle minors are positive; we call a real matrix as an $M$-matrix if it is a $Z$-matrix with $A^{-1}\geq 0$, and it is called an $H$-matrix if its comparison matrix $\langle A\rangle$ is an $M$-matrix. In Particular, an $H$-matrix with positive diagonals is called an $H_+$-matrix \cite{Bai1999}. In addition, a sufficient condition for the matrix $A$ to be a $P$-matrix is that $A$ is an $H_+$-matrix. 

Let $A = D-L-U := D-B$, where $D$, $-L$, $-U$ are the diagonal, strictly lower-triangular and the strictly upper-triangular matrices of $A$, respectively. Let $M$ is nonsingular, then $A = M-N$ is called an $M$-splitting if $M$ is an $M$-matrix and $N\geq 0$; an $H$-splitting if $\langle M \rangle - \vert N \vert$ is an $M$-matrix. Particularly, if an $H$-splitting $\langle A \rangle = \langle M \rangle - \vert N \vert$, it is called an $H$-compatible splitting.

Then, the following lemmas which are useful in the sequel are given.
\begin{lemma}[\cite{Bai1997}]\label{lemma1}
	Let $A\in \mathbb{R}^{n\times n}$ be an $H_+$-matrix, then LCP$(q, A)$ has a unique solution for any $q\in \mathbb{R}^{n\times n}$.
\end{lemma}
\begin{lemma}[\cite{Hu1982}]\label{lemma2}
	Let $B\in \mathbb{R}^{n\times n}$ be a strictly diagonal dominant matrix. Then $\forall C\in \mathbb{R}^{n\times n}$,
		\begin{gather*}
		\Vert B^{-1}C\Vert_{\infty}\leq \max_{1\leq i\leq n}\frac{(\vert C\vert e)_i}{(\langle B\rangle e)_i},
		\end{gather*}
	holds,where $e=(1,1,\cdots ,1)^T$.
\end{lemma}
\begin{lemma}[\cite{Frommer1989}]\label{lemma3}
	Let A be an H-matrix, then $\vert{A^{-1}}\vert \leq \langle A\rangle^{-1}$.
\end{lemma}
\begin{lemma}[\cite{Berman1994}]\label{lemma4}
	Let A be an M-matrix, there exists a positive diagonal matrix V, such that AV is a strictly diagonal dominant matrix with positive diagonal entries.
\end{lemma}
\begin{lemma}[\cite{Berman1994}]\label{lemma5}
	Let A, B be two Z-matrices, A be an M-matrix, and $B\geq A$. Then B is an M-matrix.
\end{lemma}
\begin{lemma}[\cite{Shen2006}]\label{lemma6}
	Let
		\begin{gather*}
		A=\begin{gathered}\begin{pmatrix} B & C \\ I & 0 \end{pmatrix}\end{gathered}\geq 0 \quad and\quad \rho(B+C)<1,
		\end{gather*}
	where I is the identity matrix. Then  $\rho(A)<1.$
\end{lemma}
\begin{lemma}[\cite{Bai2010}]\label{lemma7}
	Let $A = M - N$ be a splitting of the matrix $A\in \mathbb{R}^{n\times n}$,  $\Omega$ be a positive diagonal matrix, and $\gamma$ be a positive constant. For the LCP(q, A), the following statements hold true:
	\begin{enumerate}
		\renewcommand{\labelenumi}{\textit{(\roman{enumi})}}
		\item{if (z, r) is a solution of LCP(q, A), then $x = \frac{1}{2}\gamma (z - \Omega^{-1}r)$ satisfies the implicit fixed-point equation
				\begin{equation}\label{equation 1}
				(M+\Omega)x = Nx + (\Omega - A)\vert x\vert - \gamma q;
				\end{equation}}
		\item{if x satisfies the implicit fixed-point equation (\ref{equation 1}), then
			\begin{equation}\label{equation 2}
			z = \frac{1}{\gamma}(\vert x\vert + x) \quad and \quad r = \frac{1}{\gamma}\Omega(\vert x\vert - x)
			\end{equation}
			is a solution of the LCP(q, A).}
	\end{enumerate}
\end{lemma}
\begin{lemma}[\cite{Berman1994}]\label{lemma8}
	Let A be a Z-matrix, then the following statements are equivalent:
	\begin{enumerate}
		\renewcommand{\labelenumi}{\textit{(\roman{enumi})}}
		\item{A is an M-matrix;}
		\item{There exists a positive vector x, such that $Ax > 0$;}
		\item{Let $A = M-N$ be a splitting and $M^{-1}\geq 0$, $N \geq 0$, then $\rho(M^{-1}N) < 1$.}
	\end{enumerate}
\end{lemma}
\begin{lemma}[\cite{Zheng2013Accelerated}]\label{lemma9}
	Let $A = M_1 - N_1 = M_2 - N_2$ be two splittings of the matrix $A\in \mathbb{R}^{n\times n}$,  $\Omega_1$  and $\Omega_2$ be two positive diagonal matrices,  $\Omega$ and $\Gamma$ be $n\times n$
	positive diagonal matrices such that $\Omega = \Omega_1 + \Omega_2$. For the LCP(q, A), the following statements hold true:
	\begin{enumerate}
		\renewcommand{\labelenumi}{\textit{(\roman{enumi})}}
		\item{if (z, r) is a solution of LCP(q, A), then $x = \frac{1}{2}(\Gamma^{-1}z - \Omega^{-1}r)$ satisfies the implicit fixed-point equation
			\begin{equation}\label{equation a}
			(M_1\Gamma+\Omega_1)x = (N_1\Gamma-\Omega_2)x + (\Omega - M_2\Gamma)\vert x\vert + N_2\Gamma\vert x\vert -  q;
			\end{equation}}
		\item{if x satisfies the implicit fixed-point equation (\ref{equation a}), then
			\begin{equation}\label{equation b}
			z = {\Gamma}(\vert x\vert + x) \quad and \quad r = \Omega(\vert x\vert - x)
			\end{equation}
			is a solution of the LCP(q, A).}
	\end{enumerate}
\end{lemma}
\section{Relaxation accelerated two-sweep modulus-based matrix splitting iteration methods}
\label{sec:3}
Let $A = M - N$ be a splitting of $A$, the two-sweep modulus-based matrix splitting (TMMS) iteration method was given as follows:
\begin{method}[\cite{Wu2016} The TMMS method for LCP$(q, A)$]\label{method1}
    \quad
	\begin{enumerate}[\bf Step 1.]
		\item{Give two initial vectors $x^{(0)}, x^{(1)} \in \mathbb{R}^{n}$, then set $k=0$;}
		\item{For $k = 1, 2, \cdots$, compute $x^{(k+1)}\in \mathbb{R}^{n}$ by solving the linear system
			\begin{equation}\label{equation3}
			(\Omega + M)x^{(k+1)} = Nx^{(k)} + (\Omega - A)\vert x^{(k-1)}\vert - \gamma q.
			\end{equation}}
		\item{Set $z^{(k+1)}=\frac{1}{\gamma}(\vert x^{(k+1)}\vert + x^{(k+1)})$.} 
		\item{If the iteration sequence $\{z^{(k)}\}_{k=1}^{+\infty}$ is convergent, then stop. Otherwise, set $k := k + 1$ and return to Step 2.}
	\end{enumerate}	
\end{method}

Based on the \autoref{method1}, Ren and Wang \cite{Ren2019} extended it to the general two-sweep modulus-based matrix splitting iteration method. 

For $\Omega_1, \Omega_2$ being two positive diagonal matrices, let $A\Omega_1 = M_{\Omega_1} - N_{\Omega_1}$, $z = \Omega_1(\vert x\vert + x)$ and $r = \Omega_2(\vert x\vert - x)$, the following is the general two-sweep modulus-based matrix splitting (GTMMS) iteration method:

\begin{method}[\cite{Ren2019} The GTMMS method for LCP$(q, A)$]\label{method2}
	\quad
	\begin{enumerate}[\bf Step 1.]
		\item{Give two initial vectors $x^{(0)}, x^{(1)} \in \mathbb{R}^{n}$, then set $k=0$;}
		\item{For $k = 1, 2, \cdots$, compute $x^{(k+1)}\in \mathbb{R}^{n}$ by solving the linear system
			\begin{equation}\label{equation4}
			(\Omega_2 + M_{\Omega_1})x^{(k+1)} = N_{\Omega_1}x^{(k)} + (\Omega_2 - A\Omega_1)\vert x^{(k-1)}\vert - q.
			\end{equation}}
		\item{Set $z^{(k+1)}=\Omega_1(\vert x^{(k+1)}\vert + x^{(k+1)})$.} 
		\item{If the iteration sequence $\{z^{(k)}\}_{k=1}^{+\infty}$ is convergent, then stop. Otherwise, set $k := k + 1$ and return to Step 2.}
	\end{enumerate}	
\end{method}

To improve the computational efficiency, let $A\Omega_1=M_{\Omega_1}-N_{\Omega_1}=M_{\Omega_2}-N_{\Omega_2}$ be two splittings of $A\Omega_1$, by utilizing \autoref{lemma9}, we introduce a nonnegative relaxation parameter $\theta$ and a positive diagonal matrix $\Omega_3$. Then the relaxation accelerated two-sweep modulus-based matrix splitting (RATMMS) iteration method is established as follows:

\begin{method}[The RATMMS method for LCP$(q, A)$]\label{method3}
	\quad
	\begin{enumerate}[\bf Step 1.]
		\item{Give two initial vectors $x^{(0)}, x^{(1)} \in \mathbb{R}^{n}$ and a nonnegative relaxation parameters $\theta$, then set $k=0$;}
		\item{For $k = 1, 2, \cdots$, compute $x^{(k+1)}\in \mathbb{R}^{n}$ by solving the linear system
			\begin{gather}\label{equation5}
			\begin{split}
			(\Omega_3+\Omega_2 + M_{\Omega_1})x^{(k+1)}=&(\Omega_3+N_{\Omega_1})[\theta x^{(k)}+(1-\theta)x^{(k-1)}]\\
			&+(\Omega_2 - M_{\Omega_2})\vert x^{(k)} \vert + N_{\Omega_2} \vert x^{(k-1)} \vert - q.
			\end{split}
			\end{gather}}
		\item{Set $z^{(k+1)}=\Omega_1(\vert x^{(k+1)}\vert + x^{(k+1)})$.} 
		\item{If the iteration sequence $\{z^{(k)}\}_{k=1}^{+\infty}$ is convergent, then stop. Otherwise, set $k := k + 1$ and return to Step 2.}
	\end{enumerate}	
\end{method}
\begin{rem}
	For $\theta = 1$ and $\Omega_3=0$, \autoref{method3} reduces to the accelerated two-sweep modulus-based matrix splitting (ATMMS) iteration method.
\end{rem}
\begin{rem}
	For
	\begin{gather*}
	M_{\Omega_1} = \frac{1}{\alpha}(D_{A\Omega_1} - \beta L_{A\Omega_1}), N_{\Omega_1} = \frac{1}{\alpha}[(1-\alpha)D_{A\Omega_1} +(\alpha-\beta) L_{A\Omega_1}+\alpha U_{A\Omega_1}];\\ M_{\Omega_2} = D_{A\Omega_1} - U_{A\Omega_1}, N_{\Omega_2} = L_{A\Omega_1},
	\end{gather*}
	where $\alpha, \beta \in \mathbb{R}$ and $\alpha\not = 0$, \autoref{method3} gives the relaxation accelerated two-sweep modulus-based accelerated overrelaxation (RATMAOR) iteration method. Set $(\alpha, \beta)=(\alpha, \alpha), (1, 1), (1, 0)$,  the RATMAOR iteration method reduces to the relaxation accelerated two-sweep modulus-based successive overrelaxation (RATMSOR) iteration method, the relaxation accelerated two-sweep modulus-based  Gauss–Seidel (RATMGS) iteration method and the relaxation accelerated two-sweep modulus-based  Jacobi (RATMJ) iteration method, respectively.
\end{rem}
\section{Convergence analysis}
\label{sec:4}
In this section, the convergence analysis for Method \autoref{method3} is investigated where the system matrix $A$ of LCP$(q, A)$ is an $H_+$-matrix.

\begin{lemma}\label{lemma4.1}
	\textit{Let $A\in \mathbb{R}^{n\times n}$ be an $H_+$-matrix, $\Omega_1$, $\Omega_2$, $\Omega_3$ be given positive diagonal matrices, and $A\Omega_1 = M_{\Omega_1} - N_{\Omega_1}=M_{\Omega_2} - N_{\Omega_2}$ be two $H$-compatible splittings of $A\Omega_1$, $\tilde{\mathcal{L}}=(\Omega_3+\Omega_2 + \langle M_{\Omega_1}\rangle)^{-1}[(\theta + \vert 1-\theta \vert) \vert \Omega_3+N_{\Omega_1}\vert + \vert\Omega_2 - M_{\Omega_2}\vert + \vert N_{\Omega_2}\vert]$. Then the iteration sequence $\{z^{(k)}\}^{+\infty}_{k=0}$ generated by Method \autoref{method3} converges to the unique solution $z^*$ if  $\rho(\tilde{\mathcal{L}}) < 1$.}
\end{lemma}
\begin{proof}
	Let $(z^*, r^*)$ is the exact solution of LCP$(q, A)$, then $x^* = \frac{1}{2}(\Omega_1^{-1} z^* - \Omega_2^{-1} r^*)$, $\vert x^* \vert= \frac{1}{2}(\Omega_1^{-1} z^* + \Omega_2^{-1} r^*)$ satisfies
		\begin{gather}\label{equation6}
		\begin{split}
		(\Omega_3+\Omega_2 + M_{\Omega_1})x^* = (\Omega_3+N_{\Omega_1})[\theta x^* +(1-\theta)x^*]+(\Omega_2 - M_{\Omega_2})\vert x^* \vert + N_{\Omega_2} \vert x^* \vert - q.
		\end{split}
		\end{gather}
	Subtracting (\ref{equation6}) from (\ref{equation5}), we have
		\begin{gather*}
		\begin{split}
		&\vert x^{(k+1)}-x^* \vert \\
		=&\vert (\Omega_3+\Omega_2 + M_{\Omega_1})^{-1}(\Omega_3+N_{\Omega_1})[\theta (x^{(k)} - x^*)+(1-\theta)(x^{(k-1)} - x^*)] + (\Omega_3+\Omega_2+ M_{\Omega_1})^{-1}(\Omega_2 - M_{\Omega_2})\\
		&(\vert x^{(k)} \vert - \vert x^* \vert)+ (\Omega_3+\Omega_2 + M_{\Omega_1})^{-1}N_{\Omega_2}(\vert x^{(k-1)} \vert - \vert x^* \vert)\vert\\
		\leq &\vert (\Omega_3+\Omega_2 + M_{\Omega_1})^{-1}(\Omega_3+N_{\Omega_1})\vert(\theta \vert x^{(k)} - x^*\vert+\vert 1-\theta\vert \vert x^{(k-1)} - x^*\vert) +\vert (\Omega_3+\Omega_2 + M_{\Omega_1})^{-1}(\Omega_2 - M_{\Omega_2})\vert\\
		&\vert x^{(k)}  -  x^* \vert + \vert(\Omega_3+\Omega_2 + M_{\Omega_1})^{-1}N_{\Omega_2}\vert \vert x^{(k-1)} - x^* \vert\\
		=&[\theta \vert (\Omega_3+\Omega_2 + M_{\Omega_1})^{-1}(\Omega_3+N_{\Omega_1})\vert + \vert (\Omega_3+\Omega_2 + M_{\Omega_1})^{-1}(\Omega_2 - M_{\Omega_2})\vert]\vert x^{(k)} - x^* \vert\\
		&+[\vert 1-\theta \vert \vert (\Omega_3+\Omega_2 + M_{\Omega_1})^{-1}(\Omega_3+N_{\Omega_1})\vert +\vert (\Omega_3+\Omega_2 + M_{\Omega_1})^{-1}N_{\Omega_2}\vert]\vert x^{(k-1)} - x^*\vert.\\
		\end{split}
		\end{gather*}
	Then, let
		\begin{gather*}
		\begin{split}
		&F = \theta \vert (\Omega_3+\Omega_2 + M_{\Omega_1})^{-1}(\Omega_3+N_{\Omega_1})\vert + \vert (\Omega_3+\Omega_2 + M_{\Omega_1})^{-1}(\Omega_2 - M_{\Omega_2})\vert,\\
		&G = \vert 1-\theta \vert \vert (\Omega_3+\Omega_2 + M_{\Omega_1})^{-1}(\Omega_3+N_{\Omega_1})\vert +\vert (\Omega_3+\Omega_2 + M_{\Omega_1})^{-1}N_{\Omega_2}\vert.\\
		\end{split}
		\end{gather*}
	We obtain
		\begin{gather*}
		\Bigg\vert \begin{gathered}\begin{pmatrix} x^{(k+1)} - x^* \\ x^{(k)} - x^*\end{pmatrix}\end{gathered}\Bigg\vert
		\leq \begin{gathered}\begin{pmatrix} F & G \\ I & 0 \end{pmatrix}\end{gathered}
		\Bigg\vert \begin{gathered}\begin{pmatrix} x^{(k)} - x^* \\ x^{(k-1)} - x^*\end{pmatrix}\end{gathered}\Bigg\vert.
		\end{gather*}
	Let
		\begin{gather*}
		\mathcal{L} = \begin{gathered}\begin{pmatrix} F & G \\ I & 0 \end{pmatrix}\end{gathered},
		\end{gather*}
	evidently, the iteration sequence $\{z^{(k)}\}^{+\infty}_{k=0}$ converges to the unique solution $z^*$ if $\rho(\mathcal{L}) < 1$. Since $\mathcal{L}\geq 0$, by Lemma \ref{lemma6}, $\rho(F+G)<1$ implies $\rho(\mathcal{L})<1$. To prove the convergence of the Method \ref{method3}, it is sufficient to prove $\rho(F+G)<1$.
	
	For Method \ref{method3},
		\begin{gather*}
		\begin{split}
		F+G \leq& (\theta + \vert 1-\theta \vert) \vert (\Omega_3+\Omega_2 + M_{\Omega_1})^{-1}\vert\vert\Omega_3+N_{\Omega_1}\vert \\
		&+ \vert(\Omega_3+\Omega_2 + M_{\Omega_1})^{-1}\vert(\vert\Omega_2 - M_{\Omega_2}\vert+\vert N_{\Omega_2}\vert)\\
		\leq& \vert (\Omega_3+\Omega_2 + M_{\Omega_1})^{-1}\vert[(\theta + \vert 1-\theta \vert) \vert \Omega_3+N_{\Omega_1}\vert + \vert\Omega_2 - M_{\Omega_2}\vert+\vert N_{\Omega_2}\vert].
		\end{split}
		\end{gather*}
	
	Under the conditions that $A$ is an $H_+$-matrix and $A\Omega_1 = M_{\Omega_1} - N_{\Omega_1}$ is an $H$-compatible splitting, i.e. $\langle M_{\Omega_1}\rangle - \vert N_{\Omega_1}\vert$ is an $M$-matrix, then by \autoref{lemma5}, $\langle M_{\Omega_1}\rangle\geq \langle M_{\Omega_1}\rangle-\vert N_{\Omega_1}\vert $ implies that $M_{\Omega_1}$ is an $H$-matrix, and $\Omega_3+\Omega_2 + M_{\Omega_1}$ is also an $H$-matrix.
	
	By \autoref{lemma3}, we obtain
		\begin{gather*}
		0\leq \vert (\Omega_3+\Omega_2 + M_{\Omega_1})^{-1}\vert \leq (\Omega_3+\Omega_2 + \langle M_{\Omega_1}\rangle)^{-1},
		\end{gather*}
	hence,	
		\begin{align*}
		F+G &\leq (\Omega_3+\Omega_2 + \langle M_{\Omega_1}\rangle)^{-1}[(\theta + \vert 1-\theta \vert) \vert \Omega_3+N_{\Omega_1}\vert + \vert\Omega_2 - M_{\Omega_2}\vert+\vert N_{\Omega_2}\vert]:=\tilde{\mathcal{L}}.
		\end{align*}
	
	Hence, the iteration sequence $\{z^{(k)}\}^{+\infty}_{k=0}$ generated by Method \ref{method3} converges to the unique solution $z^*$ if  $\rho(\tilde{\mathcal{L}}) < 1$. The proof is completed.
\end{proof}

\begin{theorem}\label{theorem1}
	Under the same assumptions and notations as in \autoref{lemma4.1}, for arbitrary two initial vectors, the iteration sequence $\{z^{(k)}\}^{+\infty}_{k=0}$ generated by \autoref{method3} converges to the unique solution $z^*$ of the LCP(q, A) for one of the following conditions:
	\begin{enumerate}
		\renewcommand{\labelenumi}{\textit{(\roman{enumi})}}	
		\item{$0<\theta\leq 1$, $\Omega_1$ and $\Omega_2$ satisfy
				\begin{gather*}
				\Omega_2 e \geq D_{M_{\Omega_2}}e \mbox{ and } \frac{1}{2}(\vert M_{\Omega_2} \vert + \vert N_{\Omega_2}\vert-\langle A\Omega_1\rangle)Ve < \Omega_2 Ve < D_{M_{\Omega_2}}Ve;
				\end{gather*}}
		\item{$\theta >1$, and 
				\begin{gather*}
				\begin{cases}
				\omega < \delta^{(1)}, &\Omega_2 e \geq D_{A\Omega_1}e,\\
				\omega < \delta^{(2)}, & \frac{1}{2}(\vert M_{\Omega_2} \vert + \vert N_{\Omega_2}\vert-\langle A\Omega_1\rangle)Ve < \Omega_2 Ve < D_{M_{\Omega_2}}Ve, 
				\end{cases}
				\end{gather*}}
	\end{enumerate}
	where V is an arbitrary positive diagonal matrix such that $\langle A\Omega_1\rangle V$ is strictly diagonal dominant matrix., $i = 1, 2, \cdots, n$.
	
	For the above,
		\begin{gather*}
		\delta^{(1)} = \min_{1\leq i \leq n} \frac{(\langle A\Omega_1 \rangle Ve)_i}{[(\Omega_3 + \vert N_{\Omega_1} \vert)Ve]_i}+1, 
		\end{gather*}
		\begin{gather*}
		\delta^{(2)} = \min_{1\leq i \leq n} \frac{[(2\Omega_2+\langle A\Omega_1 \rangle - \vert M_{\Omega_2} \vert - \vert N_{\Omega_2}\vert)Ve]_i}{[2(\Omega_3 + \vert N_{\Omega_1} \vert)Ve]_i}+1.
		\end{gather*}
\end{theorem}
\begin{proof}
	By \autoref{lemma4.1}, we only need to demonstrate $\rho(\tilde{\mathcal{L}})<1$. It's easy to see that $\langle A\Omega_1\rangle$ and $\Omega_3+\Omega_2 + \langle M_{\Omega_1}\rangle$ are $M$-matrices. By \autoref{lemma4}, therefore, it's obviously that there exists a positive diagonal matrix $V$ such that $\langle A\Omega_1\rangle V$ and $(\Omega_3+\Omega_2 + \langle M_{\Omega_1}\rangle)V$ are two strictly diagonal dominant matrices. Then, by \autoref{lemma2}, we have
		\begin{gather*}
		\begin{split}
		\rho(\tilde{\mathcal{L}}) =& \rho(V^{-1}\tilde{\mathcal{L}}V)\leq {\Vert V^{-1}\tilde{\mathcal{L}}V \Vert}_\infty\\
		=& \Vert[(\Omega_3+\Omega_2 + \langle M_{\Omega_1}\rangle)V]^{-1}\\
		&[(\theta + \vert 1-\theta \vert) \vert \Omega_3+N_{\Omega_1}\vert + \vert\Omega_2 - M_{\Omega_2}\vert+\vert N_{\Omega_2}\vert]V \Vert_\infty\\
		\leq& \max_{1\leq i\leq n}\frac{\{[(\theta + \vert 1-\theta \vert) \vert \Omega_3+N_{\Omega_1}\vert + \vert\Omega_2 - M_{\Omega_2}\vert+\vert N_{\Omega_2}\vert]Ve\}_i}{[(\Omega_3+\Omega_2 + \langle M_{\Omega_1}\rangle)Ve]_i}.\\
		\end{split}
		\end{gather*}
		
	\vspace{5pt}
	(i) In this case,
		\begin{gather*}
		\rho(\tilde{\mathcal{L}})\leq \max_{1\leq i\leq n}\frac{[(\vert \Omega_3+N_{\Omega_1}\vert +\vert\Omega_2 - M_{\Omega_2}\vert+\vert N_{\Omega_2}\vert)Ve]_i}{[(\Omega_3+\Omega_2 + \langle M_{\Omega_1}\rangle)Ve]_i}.
		\end{gather*}
	If $\Omega_2\geq D_{M_{\Omega_2}}$, we have	
		\begin{gather*}
		\begin{split}
		&(\Omega_3+\Omega_2 + \langle M_{\Omega_1}\rangle)Ve - (\vert \Omega_3+N_{\Omega_1}\vert + \vert\Omega_2 - M_{\Omega_2}\vert+\vert N_{\Omega_2}\vert)Ve\\
		=&(\Omega_3+\Omega_2 + \langle M_{\Omega_1}\rangle - \vert \Omega_3+N_{\Omega_1}\vert -  \vert\Omega_2 - D_{M_{\Omega_2}}\vert-\vert B_{M_{\Omega_2}}\vert-\vert N_{\Omega_2}\vert)Ve\\
		\geq&(\Omega_3+\Omega_2 + \langle M_{\Omega_1}\rangle -  \Omega_3 -\vert N_{\Omega_1}\vert -  \Omega_2 + \langle M_{\Omega_2}\rangle - \vert N_{\Omega_2}\vert)Ve\\
		=&2\langle A\Omega_1\rangle Ve>0
		\end{split}
		\end{gather*}
	Therefore, $\max_{1\leq i\leq n}\frac{[(\vert \Omega_3+N_{\Omega_1}\vert +\vert\Omega_2 - M_{\Omega_2}\vert+\vert N_{\Omega_2}\vert)Ve]_i}{[(\Omega_3+\Omega_2 + \langle M_{\Omega_1}\rangle)Ve]_i} < 1$, thus $\rho(\tilde{\mathcal{L}}) < 1$.
	
	\vspace{5pt}
	If $\Omega_2 < D_{M_{\Omega_2}}$, we have
		\begin{gather*}
		\begin{split}
		&(\Omega_3+\Omega_2 + \langle M_{\Omega_1}\rangle)Ve - (\vert \Omega_3+N_{\Omega_1}\vert + \vert\Omega_2 - M_{\Omega_2}\vert+\vert N_{\Omega_2}\vert)Ve\\
		=&(\Omega_3+\Omega_2 + \langle M_{\Omega_1}\rangle - \vert \Omega_3+N_{\Omega_1}\vert -  \vert\Omega_2 - D_{M_{\Omega_2}}\vert-\vert B_{M_{\Omega_2}}\vert-\vert N_{\Omega_2}\vert)Ve\\
		\geq&(2\Omega_2  + \langle M_{\Omega_1}\rangle - \vert N_{\Omega_1}\vert - \vert M_{\Omega_2}\vert -\vert N_{\Omega_2}\vert)Ve\\
		=&(2\Omega_2  + \langle A\Omega_1\rangle - \vert M_{\Omega_2}\vert -\vert N_{\Omega_2}\vert)Ve.
		\end{split}
		\end{gather*}
	There exists $\Omega_2$, such that $ D_{M_{\Omega_2}}Ve>\Omega_2 Ve> \frac{1}{2}(\vert M_{\Omega_2}\vert +\vert N_{\Omega_2}\vert - \langle A\Omega_1\rangle)Ve$.
	Therefore,
	\begin{gather*}
	\max_{1\leq i\leq n}\frac{[(\vert \Omega_3+N_{\Omega_1}\vert +\vert\Omega_2 - M_{\Omega_2}\vert+\vert N_{\Omega_2}\vert)Ve]_i}{[(\Omega_3+\Omega_2 + \langle M_{\Omega_1}\rangle)Ve]_i} < 1,
	\end{gather*} 
	thus $\rho(\tilde{\mathcal{L}}) < 1$.
	
	\vspace{5pt}
	(ii) In this case,
		\begin{gather*}
		\rho(\tilde{\mathcal{L}})\leq \max_{1\leq i\leq n}\frac{\{[(2\theta - 1)\vert \Omega_3+N_{\Omega_1}\vert + \vert\Omega_2 - M_{\Omega_2}\vert+\vert N_{\Omega_2}\vert]Ve\}_i}{[(\Omega_3+\Omega_2 + \langle M_{\Omega_1}\rangle)Ve]_i}.
		\end{gather*}
	
	If $\Omega_2\geq D_{M_{\Omega_2}}$, we have 
		\begin{gather*}
		\begin{split}
		&(\Omega_3+\Omega_2 + \langle M_{\Omega_1}\rangle)Ve - [(2\theta-1)\vert \Omega_3+N_{\Omega_1}\vert + \vert\Omega_2 - M_{\Omega_2}\vert+\vert N_{\Omega_2}\vert]Ve\\
		=&[(\Omega_3+\Omega_2 + \langle M_{\Omega_1}\rangle - (2\theta-1)\vert \Omega_3+N_{\Omega_1}\vert - \vert\Omega_2 - M_{\Omega_2}\vert-\vert N_{\Omega_2}\vert]Ve\\
		\geq&[\langle M_{\Omega_1}\rangle - 2(\theta-1)\Omega_3-(2\theta-1)\vert N_{\Omega_1}\vert+\langle M_{\Omega_2}\rangle -\vert N_{\Omega_2}\vert]Ve\\
		=&[2\langle A\Omega_1\rangle - 2(\theta-1)(\Omega_3+\vert N_{\Omega_1}\vert)]Ve.
		\end{split}
		\end{gather*}
	We can set $\theta < \min_{1\leq i \leq n}\frac{(\langle A\Omega_1 \rangle Ve)_i}{(\Omega_3+\vert N_{\Omega_1}\vert)Ve]_i} + 1$,where
		\begin{gather*}
		\frac{(\langle A\Omega_1 \rangle Ve)_i}{(\Omega_3+\vert N_{\Omega_1}\vert)Ve]_i}>0,
		\end{gather*}
	then we have
		\begin{gather*}
		[2\langle A\Omega_1\rangle - 2(\theta-1)(\Omega_3+\vert N_{\Omega_1}\vert)]Ve>0,
		\end{gather*}
	which implies that $\max_{1\leq i\leq n}\frac{\{[(2\theta - 1)\vert \Omega_3+N_{\Omega_1}\vert + \vert\Omega_2 - M_{\Omega_2}\vert+\vert N_{\Omega_2}\vert]Ve\}_i}{[(\Omega_3+\Omega_2 + \langle M_{\Omega_1}\rangle)Ve]_i} < 1$, thus $\rho(\tilde{\mathcal{L}}) < 1$.
	
	If $\Omega_2 < D_{M_{\Omega_2}}$, we have
		\begin{gather*}
		\begin{split}
		&(\Omega_3+\Omega_2 + \langle M_{\Omega_1}\rangle)Ve - [(2\theta-1)\vert \Omega_3+N_{\Omega_1}\vert + \vert\Omega_2 - M_{\Omega_2}\vert+\vert N_{\Omega_2}\vert]Ve\\
		\geq&[2\Omega_2+\langle M_{\Omega_1}\rangle - 2(\theta-1)\Omega_3-(2\theta-1)\vert N_{\Omega_1}\vert-\vert M_{\Omega_2}\vert -\vert N_{\Omega_2}\vert]Ve\\
		=&[2\Omega_2+\langle A\Omega_1\rangle - 2(\theta-1)(\Omega_3+\vert N_{\Omega_1}\vert)-\vert M_{\Omega_2}\vert -\vert N_{\Omega_2}\vert]Ve.
		\end{split}
		\end{gather*}
	We can choose that $\frac{1}{2}(\vert M_{\Omega_2} \vert + \vert N_{\Omega_2}\vert-\langle A\Omega_1 \rangle)Ve < \Omega_2 Ve < D_{M_{\Omega_2}}Ve$, and set 
	\begin{gather*}
	\theta < \min_{1\leq i \leq n}\frac{[(2\Omega_2 + \langle A\Omega_1 \rangle - \vert M_{\Omega_2}\vert + \vert N_{\Omega_2}\vert)Ve]_i}{[2(\Omega_3+\vert N_{\Omega_1}\vert)Ve]_i}+1,
	\end{gather*}
	 where 
		\begin{gather*}
		{1\leq i \leq n}\frac{[(2\Omega_2 + \langle A\Omega_1 \rangle - \vert M_{\Omega_2}\vert + \vert N_{\Omega_2}\vert)Ve]_i}{[2(\Omega_3+\vert N_{\Omega_1}\vert)Ve]_i}>0,
		\end{gather*}
	then
		\begin{gather*}
		[2\Omega_2+\langle A\Omega_1\rangle - 2(\theta-1)(\Omega_3+\vert N_{\Omega_1}\vert)-\vert M_{\Omega_2}\vert -\vert N_{\Omega_2}\vert]Ve>0,
		\end{gather*}
	which implies that $\max_{1\leq i\leq n}\frac{\{[(2\theta - 1)\vert \Omega_3+N_{\Omega_1}\vert + \vert\Omega_2 - M_{\Omega_2}\vert+\vert N_{\Omega_2}\vert]Ve\}_i}{[(\Omega_3+\Omega_2 + \langle M_{\Omega_1}\rangle)Ve]_i} < 1$, thus $\rho(\tilde{\mathcal{L}}) < 1$. This complete the proof.	
\end{proof}

\begin{theorem}\label{theorem2}
	Let $A\in \mathbb{R}^{n\times n}$  be an $H_+$-matrix, and $A\Omega_1=M_{\Omega_1}-N_{\Omega_1}=M_{\Omega_2}-N_{\Omega_2}$ be two splittings where $\Omega_1$, $\Omega_2$ and $\Omega_3$ be given positive diagonal matrices. Assume that $\rho := \rho(D_{A\Omega_1}^{-1}\vert B_{A\Omega_1}\vert)$, Then for arbitrary two initial vectors, the RATMAOR iteration method is convergent for either of the following conditions:
	\begin{enumerate}[(1)]	
		\item{$0<\theta\leq 1$ and $(i)\cup (ii)\cup (iii)\cup (iv)$, where
			\vspace{5pt}
			\begin{enumerate}[(i)]
				\item{$\Omega_2\geq \frac{1}{2}D_{A\Omega_1}, 0<\alpha<\frac{2}{1+2\rho}, 0<\beta\leq \alpha, \rho<\frac{1}{2};$}
				\vspace{5pt}
				\item{$\Omega_2\geq \frac{1}{2}D_{A\Omega_1}, 2\beta\rho<\alpha<2-2\beta\rho, 0<\alpha\leq \beta, \rho<\frac{1}{2\beta};$}
				\vspace{5pt}
				\item{$\Omega_2 \geq D_{A\Omega_1}, 0<\alpha<\frac{1}{\rho}, 0<\beta\leq \alpha, \rho<1;$}
				\vspace{5pt}
				\item{$\Omega_2 \geq D_{A\Omega_1}, \beta\rho<\alpha, 0<\alpha\leq \beta, \rho<\frac{1}{\beta}.$}
				\vspace{5pt}
		\end{enumerate}}
		
		\item{$\theta>1$ and $(i)\cup (ii)$, where
			\vspace{5pt}
			\begin{enumerate}[(i)]
				\item{$\Omega_2 \geq \frac{1}{2}D_{A\Omega_1}+(\theta-1)\Omega_3, \Omega_3\leq \frac{1}{2(\theta-1)}D_{A\Omega_1}, \frac{2(\theta-1)}{2\theta(1-\rho)-1}<\alpha<\frac{2\theta}{2\theta(1+\rho)-1}, 0<\beta\leq \alpha, \rho<\frac{1}{2\omega};$}
				\vspace{5pt}
				\item{$\Omega_2 \geq \frac{1}{2}D_{A\Omega_1}+(\theta-1)\Omega_3, \Omega_3\leq \frac{1}{2(\theta-1)}D_{A\Omega_1}, \frac{2(\theta\beta\rho+\theta-1)}{2\theta-1}<\alpha<\frac{2\theta(1-\beta\rho)}{2\theta-1}, 0<\alpha\leq \beta, \rho<\frac{1}{2\theta\beta}.$}
				\vspace{5pt}
		\end{enumerate}} 
	\end{enumerate}
\end{theorem}
\begin{proof}
	As is the proof of \autoref{lemma4.1}, consider $\hat{M}=\Omega_3+\Omega_2+\langle M_{\Omega_1}\rangle$, $\hat{N}=( \theta+\vert 1-\theta \vert)\vert \Omega_3+N_{\Omega_1}\vert + \vert \Omega_2 - M_{\Omega_2} \vert + \vert N_{\Omega_2}\vert$, $\hat{A}=\hat{M}-\hat{N}$, and $\tilde{\mathcal{L}}=\hat{M}^{-1}\hat{N}$. By \autoref{lemma8}, if $\hat{A}$ and $\hat{M}$ are $M$-matrices, $\hat{N}\geq 0$, then $\rho(\tilde{\mathcal{L}})=\rho(\hat{M}^{-1}\hat{N})<1$. Also note that RATMAOR iteration method has two splittings $A\Omega_1= M_{\Omega_1}-N_{\Omega_1}=M_{\Omega_2}-N_{\Omega_2}$ where
		\begin{align*}
		&M_{\Omega_1} = \frac{1}{\alpha}(D_{A\Omega_1} - \beta L_{A\Omega_1}), N_{\Omega_1} = \frac{1}{\alpha}[(1-\alpha)D_{A\Omega_1} +(\alpha-\beta) L_{A\Omega_1}+\alpha U_{A\Omega_1}];\\
		&M_{\Omega_2}=D_{A\Omega_1}-U_{A\Omega_1}, N_{\Omega_2}= L_{A\Omega_1}.
		\end{align*}
	
	(1) For $0<\theta\leq 1$, by some simple calculations we have that
		\begin{align*}
		\hat{M}=& \Omega_3+\Omega_2+\frac{1}{\alpha}(D_{A\Omega_1}-\beta \vert L_{A\Omega_1}\vert),\\
		\hat{N}=& \vert \Omega_3+N_{\Omega_1}\vert+\vert \Omega_2-M_{\Omega_2} \vert+\vert N_{\Omega_2}\vert \\
		\leq& \Omega_3+\vert N_{\Omega_1}\vert+\vert \Omega_2-M_{\Omega_2} \vert+\vert N_{\Omega_2}\vert\\
		=& \Omega_3+\vert \frac{1}{\alpha}[(1-\alpha)D_{A\Omega_1}+(\alpha-\beta)L_{A\Omega_1}+\alpha U_{A\Omega_1}]\vert \\
		&+ \vert \Omega_2-(D_{A\Omega_1}- U_{A\Omega_1})\vert + \vert L_{A\Omega_1}\vert\\
		=&\Omega_3+\frac{\vert 1-\alpha \vert}{\alpha}D_{A\Omega_1}+\frac{\alpha+\vert \alpha-\beta \vert}{\alpha}\vert L_{A\Omega_1}\vert + 2\vert U_{A\Omega_1}\vert\\
		&+ \vert \Omega_2-D_{A\Omega_1}\vert := \tilde{N},
		\end{align*}
	and
		\begin{align}\label{equation7}
		\nonumber\hat{A}=&\hat{M}-\hat{N}\geq \hat{M}-\tilde{N}\\
		\nonumber=& \Omega_2-\vert\Omega_2-D_{A\Omega_1}\vert+\frac{1-\vert 1-\alpha\vert}{\alpha}D_{A\Omega_1}\\
		&-\frac{\alpha+\beta+\vert \alpha-\beta \vert}{\alpha}\vert L_{A\Omega_1}\vert-2\vert U_{A\Omega_1}\vert.
		\end{align}
	Obviously, $\hat{M}$ is an $M$-matrix, $\hat{N}\geq 0$, hence, it suffices to prove that $\hat{A}$ is an $M$-matrix.
	
	\vspace{2pt}
	For $\Omega_2\geq \frac{1}{2}D_{A\Omega_1}$, it holds that
		\begin{gather}\label{equation**}
		\Omega_2-\vert\Omega_2-D_{A\Omega_1}\vert\geq 0,
		\end{gather}
	then, for $0<\beta\leq\alpha$, from (\ref{equation7}) we obtain that
		\begin{align}\label{equation13}
		\nonumber\hat{A}\geq& \frac{1-\vert 1-\alpha \vert}{\alpha}D_{A\Omega_1}-2\vert L_{A\Omega_1}\vert -2\vert U_{A\Omega_1}\vert\\
		\geq&\frac{1-\vert 1-\alpha \vert}{\alpha}D_{A\Omega_1}-2\vert B_{A\Omega_1} \vert:= T.
		\end{align}
	It is easy to see that $T$ is an $M$-matrix if and only if
		\begin{equation*}
		1-\vert 1-\alpha \vert>0\mbox{ and }\rho<\frac{1-\vert 1-\alpha \vert}{2\alpha},
		\end{equation*}
	which is equivalent to
		\begin{gather*}
		0<\alpha<\frac{2}{1+2\rho}\mbox{ and } \rho<\frac{1}{2}.
		\end{gather*}
	On the basis of (\ref{equation**}), if $\beta\geq\alpha>0$, from (\ref{equation7}), we obtain that
		\begin{align}\label{equation14}
		\nonumber\hat{A}\geq& \frac{1-\vert 1-\alpha \vert}{\alpha}D_{A\Omega_1}-\frac{2\beta}{\alpha}\vert L_{A\Omega_1}\vert -2\vert U_{A\Omega_1}\vert\\
		\geq& \frac{1-\vert 1-\alpha \vert}{\alpha}D_{A\Omega_1}-\frac{2\beta}{\alpha}\vert B_{A\Omega_1}\vert:= T.
		\end{align}
	Here, $T$ is an $M$-matrix if and only if
		\begin{equation*}
		1-\vert 1-\alpha \vert>0\mbox{ and } \rho<\frac{1-\vert 1-\alpha \vert}{2\beta},
		\end{equation*}
	which is equivalent to
		\begin{gather*}
		2\beta\rho<\alpha<2-2\beta\rho\mbox{ and } \rho<\frac{1}{2\beta}.
		\end{gather*}
	
	Specifically, on the basis of (\ref{equation**}), if $\Omega_2\geq D_{A\Omega_1}$, and $0<\beta\leq\alpha$, then
		\begin{equation*}
		\Omega_2-\vert\Omega_2-D_{A\Omega_1}\vert=D_{A\Omega_1}\geq 0,
		\end{equation*}
	{\setlength\abovedisplayskip{-4pt}
		\setlength\belowdisplayskip{0pt}
		\begin{align}\label{equation15}
		\nonumber\hat{A}\geq& \frac{\alpha+1-\vert 1-\alpha \vert}{\alpha}D_{A\Omega_1}-2\vert L_{A\Omega_1}\vert -2\vert U_{A\Omega_1}\vert\\
		=&\frac{\alpha+1-\vert 1-\alpha \vert}{\alpha}D_{A\Omega_1}-2\vert B_{A\Omega_1} \vert:= T.
		\end{align}}\\
	As a result, $T$ is an $M$-matrix if and only if
		\begin{equation*}
		1-\vert 1-\alpha \vert>0\mbox{ and } \rho<\frac{\alpha+1-\vert 1-\alpha \vert}{2\alpha},
		\end{equation*}
	which is equivalent to
		\begin{gather*}
		0<\alpha<\frac{1}{\rho}\mbox{ and } \rho<1.
		\end{gather*}
	If $\beta\geq\alpha>0$, from (\ref{equation7}), we obtain that
		\begin{align}\label{equation16}
		\nonumber\hat{A}\geq& \frac{\alpha+1-\vert 1-\alpha \vert}{\alpha}D_{A\Omega_1}-\frac{2\beta}{\alpha}\vert L_{A\Omega_1}\vert -2\vert U_{A\Omega_1}\vert\\
		\geq&\frac{\alpha+1-\vert 1-\alpha \vert}{\alpha}D_{A\Omega_1}-\frac{2\beta}{\alpha}\vert B_{A\Omega_1} \vert:= T.
		\end{align}
	Then, $T$ is an $M$-matrix if and only if
		\begin{equation*}
		\alpha+1-\vert 1-\alpha \vert>0\mbox{ and } \rho<\frac{\alpha+1-\vert 1-\alpha \vert}{2\beta},
		\end{equation*}
	which is equivalent to
		\begin{gather*}
		\beta\rho<\alpha\mbox{ and } \rho<\frac{1}{\beta}.
		\end{gather*}
	
	(2) For $\theta>1$, by some simple calculations we have that
		\begin{align*}
		\hat{M}=& \Omega_3+\Omega_2+\frac{1}{\alpha}(D_{A\Omega_1}-\beta \vert L_{A\Omega_1}\vert),\\
		\hat{N}=& (2\theta-1)\vert \Omega_3+N_{\Omega_1}\vert+\vert \Omega_2-M_{\Omega_2} \vert+\vert N_{\Omega_2}\vert \\
		\leq& (2\theta-1)\Omega_3+(2\theta-1)\vert N_{\Omega_1}\vert+\vert \Omega_2-M_{\Omega_2} \vert+\vert N_{\Omega_2}\vert\\
		=& (2\theta-1)\Omega_3+\vert \frac{(2\theta-1)}{\alpha}[(1-\alpha)D_{A\Omega_1}+(\alpha-\beta)L_{A\Omega_1}+\alpha U_{A\Omega_1}]\vert \\
		&+ \vert \Omega_2-(D_{A\Omega_1}- U_{A\Omega_1})\vert + \vert L_{A\Omega_1}\vert\\
		=&(2\theta-1)\Omega_3+\frac{(2\theta-1)\vert 1-\alpha \vert}{\alpha}D_{A\Omega_1}+\frac{\alpha+(2\theta-1)\vert \alpha-\beta \vert}{\alpha}\vert L_{A\Omega_1}\vert \\&
		+ 2\theta\vert U_{A\Omega_1}\vert + \vert \Omega_2-D_{A\Omega_1}\vert := \tilde{N},
		\end{align*}
	and
		\begin{align}\label{equation17}
		\nonumber\hat{A}=&\hat{M}-\hat{N}\geq \hat{M}-\tilde{N}\\
		\nonumber=& 2(1-\theta)\Omega_3+\Omega_2-\vert\Omega_2-D_{A\Omega_1}\vert+\frac{1-(2\theta-1)\vert 1-\alpha\vert}{\alpha}D_{A\Omega_1}\\
		&-\frac{\alpha+\beta+(2\theta-1)\vert \alpha-\beta \vert}{\alpha}\vert L_{A\Omega_1}\vert-2\theta\vert U_{A\Omega_1}\vert.
		\end{align}
	Similarly, to prove that $\hat{A}$ is an $M$-matrix, we give the follow discussion.
	
	\vspace{5pt}
	For $\Omega_2\geq\frac{1}{2}D_{A\Omega_1}+(\theta-1)\Omega_3$, $\Omega_3\leq\frac{1}{2(\theta-1)}D_{A\Omega_1}$ and $0<\beta\leq\alpha$, it holds that 
		\begin{gather}\label{equation***}
		2(1-\theta)\Omega_3+\Omega_2-\vert\Omega_2-D_{A\Omega_1}\vert\geq 0,
		\end{gather}
	therefore, from (\ref{equation17}) we obtain that
		\begin{align}\label{equation18}
		\nonumber\hat{A}\geq& \frac{1-(2\theta-1)\vert 1-\alpha \vert}{\alpha}D_{A\Omega_1}-\frac{2\theta\alpha-2(\theta-1)\beta}{\alpha}\vert L_{A\Omega_1}\vert -2\theta\vert U_{A\Omega_1}\vert\\
		\geq&\frac{1-(2\theta-1)\vert 1-\alpha \vert}{\alpha}D_{A\Omega_1}-2\theta\vert B_{A\Omega_1}\vert\:= T.
		\end{align}
	Then, $T$ is an $M$-matrix if and only if
		\begin{equation*}
		1-(2\theta-1)\vert 1-\alpha \vert>0\mbox{ and } \rho<\frac{1-(2\theta-1)\vert 1-\alpha \vert}{2\theta\alpha},
		\end{equation*}
	which is equivalent to
		\begin{gather*}
		\frac{2(\theta-1)}{2\theta(1-\rho)-1}<\alpha<\frac{2\theta}{2\theta(1+\rho)-1}\mbox{ and } \rho<\frac{1}{2\theta}.
		\end{gather*}
	On the basis of (\ref{equation***}), if $\beta\geq\alpha>0$, from (\ref{equation17}), we obtain that
		\begin{align}\label{equation19}
		\nonumber\hat{A}\geq& \frac{1-(2\theta-1)\vert 1-\alpha \vert}{\alpha}D_{A\Omega_1}-\frac{2\theta\beta-2(\theta-1)\alpha}{\alpha}\vert L_{A\Omega_1}\vert -2\theta\vert U_{A\Omega_1}\vert\\
		\geq&\frac{1-(2\theta-1)\vert 1-\alpha \vert}{\alpha}D_{A\Omega_1}-\frac{2\theta\beta}{\alpha}\vert B_{A\Omega_1}\vert\:= T.
		\end{align}
	Obviously, $T$ is an $M$-matrix if and only if
		\begin{equation*}
		1-(2\theta-1)\vert 1-\alpha \vert>0\mbox{ and } \rho<\frac{1-(2\theta-1)\vert 1-\alpha \vert}{2\theta\beta},
		\end{equation*}
	which is equivalent to
		\begin{align*}
		\frac{2(\theta\beta\rho+\theta-1)}{2\theta-1}<\alpha<\frac{2\theta(1-\beta\rho)}{2\theta-1}\mbox{ and } \rho<\frac{1}{2\theta\beta}.
		\end{align*}
	The proof is completed.
\end{proof}
\begin{col}
	Under the same assumptions and notations as those in \autoref{theorem2}, when $\alpha=\beta$, the RATMSOR iteration method is convergent for either of the following conditions:
	\begin{enumerate}[(1)]	
	\item{$0<\theta\leq 1$ and $(i)\cup (ii)\cup (iii)\cup (iv)$, where
		\vspace{5pt}
		\begin{enumerate}[(i)]
			\item{$\Omega_2\geq \frac{1}{2}D_{A\Omega_1}, 0<\alpha<\frac{2}{1+2\rho}, 0<\beta\leq \alpha, \rho<\frac{1}{2};$}
			\vspace{5pt}
			\item{$\Omega_2 \geq D_{A\Omega_1}, 0<\alpha<\frac{1}{\rho}, 0<\beta\leq \alpha, \rho<1;$}
	\end{enumerate}}
	
	\item{$\theta>1$ and $(i)\cup (ii)$, where
		
			\vspace{5pt}
			$\Omega_2 \geq \frac{1}{2}D_{A\Omega_1}+(\theta-1)\Omega_3, \Omega_3\leq \frac{1}{2(\theta-1)}D_{A\Omega_1}, \frac{2(\theta-1)}{2\theta(1-\rho)-1}<\alpha<\frac{2\theta}{2\theta(1+\rho)-1}, 0<\beta\leq \alpha, \rho<\frac{1}{2\omega};$}
\end{enumerate}
\end{col}	
\section{Numerical experiments}
\label{sec:5}
In this section, two numerical examples will be given to illustrate the efficiency of the relaxation accelerated two-sweep modulus-based matrix splitting iteration method in terms of iteration steps (IT), the elapsed CPU time in seconds (CPU), and norm of absolute residual vectors (RES), respectively. Here, ‘RES’   is defined as
\begin{equation*}
RES(z^{(k)}):= \Vert \min{(Az^{(k)}+q, z^{(k)})}\Vert_2,\\
\end{equation*}
where $z^{(k)}$ is the $k$th approximate solution to the LCP$(q, A)$ while the minimum is taken componentwise.

All the initial vectors in our numerical experiments are chosen as 
\begin{equation*}
x^{(0)}=x^{(1)}=(0, 0, \dots, 0)^{T}\in \mathbb{R}^{n},\\
\end{equation*}
in addition, all the computations were run in MATLAB (R2016a) on an Intel(R) Core(TM), i5 where the CPU is 3.20 GHz, the memory is 8.00 GB. Besides, all iterations are terminated either $RES(z^{(k)}) \leq 10^{-5}$ or the maximum number of iterations exceeds 500. We use ‘-’ in the table to denote that neither of the above conditions is satisfied. Then, we take $\Omega_1=kI, \Omega_2=\frac{1}{2\alpha}D_{A}$ for the relaxation accelerated two-sweep modulus-based successive overrelaxation (RATMSOR) method, the accelerated two-sweep modulus-based successive overrelaxation (ATMSOR) method, the general two-sweep modulus-based successive overrelaxation (GTMSOR) method and the general modulus-based successive overrelaxation (GMSOR) method, where ‘$\alpha$’ denotes the relaxation factor. In the subsequent numerical experiments, we choose $k=0.8$ for \autoref{example1} and $k=1$ for \autoref{example2}. Note that, when $\alpha=1$, the SOR iteration methods mentioned above reduce to the GS iteration methods.
\begin{table}[!t]
	\caption{Numerical results for RATMAOR with $m=200, \mu=2, \alpha=1 $ of \autoref{example1}.}
	\label{tab:1}       
	%
	%
	\begin{tabular}{p{1.7cm}p{0.8cm}p{1.2cm}p{1.2cm}p{1.2cm}p{1.2cm}p{1.2cm}p{1.2cm}p{1.2cm}}
		\hline\noalign{\smallskip}
		RATMAOR & $\theta$ & 0.9 &1.1 &1.3 &1.5 &1.6 &1.7 &1.8\\
		\noalign{\smallskip}\bottomrule\noalign{\smallskip}
		\multirow{3}*{$\Omega_3=0$}&IT &40	&38	&47	&71	&92	&128	&206\\
		&CPU &0.2222	&0.1694	&0.2185	&0.2832	&0.3668	&0.4725	&0.7435\\
		&RES &9.86e-06	&8.51e-06	&9.44e-06	&9.39e-06	&9.57e-06	&9.75e-06	&9.59e-06\\
		\noalign{\smallskip}\hline\noalign{\smallskip}
		\multirow{3}*{$\Omega_3=\frac{1}{4}D_A$}&IT &49	&45	&40	&34	&32	&34	&37\\
		&CPU &0.2888	&0.2378	&0.2008	&0.1810	&0.1645	&0.1736	&0.1846\\
		&RES &8.57e-06	&6.68e-06	&6.70e-06	&8.64e-06	&6.55e-06	&7.55e-06	&9.88e-06\\
		\noalign{\smallskip}\hline\noalign{\smallskip}
		\multirow{3}*{$\Omega_3=\frac{1}{2}D_A$}&IT &58	&51	&44	&36	&31	&26	&29\\
		&CPU &0.3276	&0.2531	&0.2149	&0.1878	&0.1640	&0.1420	&0.1557\\
		&RES &8.26e-06	&9.22e-06	&8.56e-06	&7.23e-06	&6.07e-06	&9.30e-06	&9.66e-06\\
		\noalign{\smallskip}\hline\noalign{\smallskip}
		\multirow{3}*{$\Omega_3=D_A$}&IT &76	&65	&54	&40	&28	&32	&35\\
		&CPU &0.3978	&0.3077	&0.2672	&0.2036	&0.1451	&0.1629	&0.1763\\
		&RES &8.48e-06	&9.58e-06	&8.36e-06	&8.53e-06	&7.69e-06	&6.14e-06	&8.15e-06\\
		\noalign{\smallskip}\hline\noalign{\smallskip}
		
	\end{tabular}
\end{table}
\begin{table}[!t]
	\caption{Numerical results for \autoref{example1} with $\alpha=1, \Omega_3=\frac{1}{2}D_A, \theta=1.7$.}
	\label{tab:2}       
	%
	%
	\begin{tabular}{p{1.1cm}p{1.7cm}p{1cm}p{1.4cm}p{1.4cm}p{1.4cm}p{1.4cm}p{1.4cm}}
		\hline\noalign{\smallskip}
		&  & $m$ & 30 & 60 & 100 &150 &200 \\
		\noalign{\smallskip}\bottomrule\noalign{\smallskip}
		\multirow{12}*{$\mu$=1.5}
		&\multirow{3}*{GMSOR}&IT &40	&42	&43	&44	&44\\
		& &CPU &0.0040	&0.0124	&0.0337	&0.1605	&0.2356\\
		& &RES &9.98e-06	&8.46e-06	&8.33e-06	&7.62e-06	&8.89e-06\\
		&\multirow{3}*{GTMSOR}&IT &45	&46	&47	&48	&49\\
		& &CPU &0.0036	&0.0140	&0.0366	&0.1756	&0.2669\\
		& &RES &7.00e-06	&9.31e-06	&9.06e-06	&9.10e-06	&8.28e-06\\
		&\multirow{3}*{ATMSOR}&IT &42	&45	&46	&48	&48\\
		& &CPU &0.0039	&0.0137	&0.0360	&0.1373	&0.2142\\
		& &RES &9.18e-06	&7.48e-06	&9.31e-06	&6.97e-06	&9.46e-06\\
		&\multirow{3}*{RATMSOR}&IT &30	&32	&33	&34	&34\\
		& &CPU &0.0035	&0.0122	&0.0318	&0.0992	&0.1821\\
		& &RES &9.20e-06	&7.29e-06	&7.59e-06	&6.95e-06	&9.42e-06\\
		\noalign{\smallskip}\hline\noalign{\smallskip}
		\multirow{12}*{$\mu$=2.5}
		&\multirow{3}*{GMSOR}&IT &32	&33	&33	&34	&34\\
		& &CPU &0.0031	&0.0106	&0.0277	&0.1212	&0.2055\\
		& &RES &7.08e-06	&7.22e-06	&9.60e-06	&7.94e-06	&9.24e-06\\
		&\multirow{3}*{GTMSOR}&IT &39	&40	&41	&41	&42	\\
		& &CPU &0.0031	&0.0125	&0.0328	&0.1554	&0.2365\\
		& &RES &8.47e-06	&8.67e-06	&7.47e-06	&9.16e-06	&8.84e-06\\
		&\multirow{3}*{ATMSOR}&IT &30	&31	&32	&33	&34\\
		& &CPU &0.0028	&0.0110	&0.0268	&0.0847	&0.1848\\
		& &RES &8.14e-06	&9.33e-06	&4.82e-06	&5.44e-06	&6.00e-06\\
		&\multirow{3}*{RATMSOR}&IT &24	&24	&24	&25	&25\\
		& &CPU &0.0027	&0.0104	&0.0260	&0.0686	&0.1738\\
		& &RES &9.43e-06	&9.61e-06	&9.84e-06	&6.32e-06	&6.45e-06\\
		\noalign{\smallskip}\hline\noalign{\smallskip}
		\multirow{12}*{$\mu$=4}
		&\multirow{3}*{GMSOR}&IT &25	&26	&27	&27	&27\\
		& &CPU &0.0023	&0.0087	&0.0238	&0.0922	&0.1876\\
		& &RES &8.60e-06	&7.62e-06	&5.95e-06	&7.37e-06	&8.56e-06\\
		&\multirow{3}*{GTMSOR}&IT &34	&36	&36	&36	&36\\
		& &CPU &0.0028	&0.0114	&0.0289	&0.1191	&0.2093\\
		& &RES &7.46e-06	&4.63e-06	&6.00e-06	&7.36e-06	&8.51e-06\\
		&\multirow{3}*{ATMSOR}&IT &22	&23	&24	&25	&25\\
		& &CPU &0.0022	&0.0085	&0.0221	&0.0617	&0.1640\\
		& &RES &8.84e-06	&8.91e-06	&7.30e-06	&5.27e-06	&7.08e-06\\
		&\multirow{3}*{RATMSOR}&IT &22	&22	&22	&22	&23\\
		& &CPU &0.0027	&0.0087	&0.0239	&0.0570	&0.1599\\
		& &RES &5.73e-06	&6.61e-06	&7.76e-06	&9.18e-06	&6.99e-06\\
		\noalign{\smallskip}\hline\noalign{\smallskip}
	\end{tabular}
\end{table}
\begin{table}[!t]
	\caption{Numerical results for \autoref{example1} with $\Omega_3=\frac{1}{2}D_A, \mu=2, \theta=1.7$.}
	\label{tab:3}       
	%
	%
	\begin{tabular}{p{1.1cm}p{1.7cm}p{1cm}p{1.4cm}p{1.4cm}p{1.4cm}p{1.4cm}p{1.4cm}}
		\hline\noalign{\smallskip}
		&  & $\alpha$ &0.6 &0.9 &1.1 &1.3 &1.5\\
		\noalign{\smallskip}\bottomrule\noalign{\smallskip}
		\multirow{9}*{$m=60$}
		&\multirow{3}*{GMSOR}&IT &51	&28	&73 &- &- \\
		& &CPU &0.0164	&0.0100	&0.0217 &- &- \\
		& &RES &7.78e-06	&7.08e-06	&9.82e-06 &- &- \\
		&\multirow{3}*{GTMSOR}&IT &73	&41	&55	&110 &-\\
		& &CPU &0.0222	&0.0129	&0.0167	&0.0315	&-\\
		& &RES &8.28e-06	&9.19e-06	&8.51e-06	&9.60e-06 &-\\
		&\multirow{3}*{RATMSOR}&IT &34	&25	&29	&40	&63\\
		& &CPU &0.0127	&0.0096	&0.0107	&0.0142	&0.0214\\
		& &RES &7.23e-06	&7.08e-06	&9.20e-06	&8.14e-06	&9.34e-06\\
		\noalign{\smallskip}\hline\noalign{\smallskip}
		\multirow{9}*{$m=150$}
		&\multirow{3}*{GMSOR}&IT &54	&30	&77 &- &-\\
		& &CPU &0.1786	&0.1020	&0.2183 &- &- \\
		& &RES &7.94e-06	&5.86e-06	&8.47e-06 &- &- \\
		&\multirow{3}*{GTMSOR}&IT &77	&44	&57	&113 &-\\
		& &CPU &0.2547	&0.1266	&0.1633	&0.2980	&-\\
		& &RES &9.12e-06	&7.59e-06	&8.62e-06	&8.99e-06 &-\\
		&\multirow{3}*{RATMSOR}&IT &36	&25	&29	&40	&63	\\
		& &CPU &0.1001	&0.0846	&0.0775	&0.0962	&0.1461\\
		& &RES &7.52e-06	&9.24e-06	&9.21e-06	&8.14e-06	&9.34e-06\\
		\noalign{\smallskip}\hline\noalign{\smallskip}
	\end{tabular}
\end{table}
\begin{table}[!t]
	\caption{Numerical results for RATMAOR with $m=200, \mu=2, \alpha=1 $ of \autoref{example2}.}
	\label{tab:4}       
	%
	%
	\begin{tabular}{p{1.7cm}p{0.8cm}p{1.2cm}p{1.2cm}p{1.2cm}p{1.2cm}p{1.2cm}p{1.2cm}p{1.2cm}}
		\hline\noalign{\smallskip}
		RATMAOR & $\theta$ &1.1 &1.3 &1.5 &1.7 &1.9 &2.1 &2.3\\
		\noalign{\smallskip}\bottomrule\noalign{\smallskip}
		\multirow{3}*{$\Omega_3=0$}&IT &64	&53	&53	&58	&66	&80	&111\\
		&CPU &0.3121	&0.2298	&0.2169	&0.2424	&0.2742	&0.3154	&0.4189\\
		&RES &8.25e-06	&7.92e-06	&8.54e-06	&9.46e-06	&9.46e-06	&9.57e-06	&9.52e-06\\
		\noalign{\smallskip}\hline\noalign{\smallskip}
		\multirow{3}*{$\Omega_3=\frac{1}{4}D_A$}&IT &43	&40	&37	&33	&30	&29	&33\\
		&CPU &0.2544	&0.2219	&0.1906	&0.1720	&0.1552	&0.1589	&0.1646\\
		&RES &7.00e-06	&7.06e-06	&6.72e-06	&9.89e-06	&8.11e-06	&1.00e-05	&7.51e-06\\
		\noalign{\smallskip}\hline\noalign{\smallskip}
		\multirow{3}*{$\Omega_3=\frac{1}{2}D_A$}&IT &48	&43	&38	&32	&27	&30	&35\\
		&CPU &0.2838	&0.2154	&0.1881	&0.1610	&0.1410	&0.1545	&0.1719\\
		&RES &7.41e-06	&8.12e-06	&7.79e-06	&9.31e-06	&9.95e-06	&9.49e-06	&8.11e-06\\
		\noalign{\smallskip}\hline\noalign{\smallskip}
		\multirow{3}*{$\Omega_3=D_A$}&IT &58	&50	&41	&31	&38	&49	&67\\
		&CPU &0.3242	&0.2522	&0.2054	&0.1683	&0.1927	&0.2255	&0.2970\\
		&RES &9.57e-06	&8.93e-06	&8.15e-06	&9.99e-06	&8.18e-06	&8.46e-06	&9.84e-06\\
		\noalign{\smallskip}\hline\noalign{\smallskip}
		
	\end{tabular}
\end{table}
\begin{table}[!t]
	\caption{Numerical results for \autoref{example2} with $\alpha=1, \Omega_3=\frac{1}{2}D_A, \theta=1.9$.}
	\label{tab:5}       
	%
	%
	\begin{tabular}{p{1.1cm}p{1.7cm}p{1cm}p{1.4cm}p{1.4cm}p{1.4cm}p{1.4cm}p{1.4cm}}
		\hline\noalign{\smallskip}
		&  & $m$ & 30 & 60 & 100 &150 &200 \\
		\noalign{\smallskip}\bottomrule\noalign{\smallskip}
		\multirow{9}*{$\mu$=1.5}
		&\multirow{3}*{GMSOR}&IT &32	&33	&34	&34	&35\\
		& &CPU &0.0027	&0.0108	&0.0302	&0.1202	&0.2136\\
		& &RES &8.97e-06	&9.04e-06	&8.07e-06	&9.98e-06	&7.87e-06\\
		&\multirow{3}*{GTMSOR}&IT &49	&50	&51	&51	&53\\
		& &CPU &0.0037	&0.0149	&0.0391	&0.1795	&0.3024\\
		& &RES &7.16e-06	&9.58e-06	&7.69e-06	&9.45e-06	&6.25e-06\\
		&\multirow{3}*{RATMSOR}&IT &30	&32	&33	&34	&34\\
		& &CPU &0.0034	&0.0125	&0.0340	&0.1101	&0.1806\\
		& &RES &9.45e-06	&6.41e-06	&6.62e-06	&6.12e-06	&8.35e-06\\
		\noalign{\smallskip}\hline\noalign{\smallskip}
		\multirow{9}*{$\mu$=2.5}
		&\multirow{3}*{GMSOR}&IT &28	&29	&29	&30	&30\\
		& &CPU &0.0028	&0.0095	&0.0273	&0.1036	&0.1975\\
		& &RES &8.10e-06	&7.58e-06	&9.94e-06	&7.79e-06	&9.03e-06\\
		&\multirow{3}*{GTMSOR}&IT &44	&46	&47	&47	&48\\
		& &CPU &0.0036	&0.0136	&0.0363	&0.1625	&0.2819\\
		& &RES &9.18e-06	&7.72e-06	&7.72e-06	&9.51e-06	&8.31e-06\\
		&\multirow{3}*{RATMSOR}&IT &25	&26	&26	&26	&26\\
		& &CPU &0.0028	&0.0105	&0.0260	&0.0688	&0.1637\\
		& &RES &8.55e-06	&5.31e-06	&6.21e-06	&7.23e-06	&8.17e-06\\
		\noalign{\smallskip}\hline\noalign{\smallskip}
		\multirow{9}*{$\mu$=4}
		&\multirow{3}*{GMSOR}&IT &24	&25	&26	&26	&26\\
		& &CPU &0.0022	&0.0089	&0.0241	&0.0991	&0.1762\\
		& &RES &9.63e-06	&8.24e-06	&6.31e-06	&7.79e-06	&9.02e-06\\
		&\multirow{3}*{GTMSOR}&IT &40	&42	&42	&42	&44\\
		& &CPU &0.0034	&0.0130	&0.0331	&0.1468	&0.2420\\
		& &RES &7.31e-06	&5.80e-06	&7.54e-06	&9.28e-06	&5.91e-06\\
		&\multirow{3}*{RATMSOR}&IT &24	&24	&24	&25	&26\\
		& &CPU &0.0026	&0.0098	&0.0258	&0.0797	&0.1736\\
		& &RES &4.23e-06	&5.66e-06	&7.60e-06	&9.09e-06	&3.94e-06\\
		\noalign{\smallskip}\hline\noalign{\smallskip}
	\end{tabular}
\end{table}
\begin{table}[!t]
	\caption{Numerical results for \autoref{example2} with $\Omega_3=\frac{1}{2}D_A, \mu=2, \theta=1.7$.}
	\label{tab:6}       
	%
	%
	\begin{tabular}{p{1.1cm}p{1.7cm}p{1cm}p{1.4cm}p{1.4cm}p{1.4cm}p{1.4cm}p{1.4cm}}
		\hline\noalign{\smallskip}
		&  & $\alpha$ &0.6 &0.9 &1.2 &1.4 &1.5\\
		\noalign{\smallskip}\bottomrule\noalign{\smallskip}
		\multirow{9}*{$m=60$}
		&\multirow{3}*{GMSOR}&IT &41	&21	&124 &- &- \\
		& &CPU &0.0123	&0.0076	&0.0348 &- &- \\
		& &RES &8.82e-06	&5.31e-06	&9.92e-06 &- &- \\
		&\multirow{3}*{GTMSOR}&IT &58	&38	&78	&184 &-\\
		& &CPU &0.0169	&0.0120	&0.0227	&0.0516	 &-\\
		& &RES &7.52e-06	&9.81e-06	&9.51e-06	&9.91e-06 &-\\
		&\multirow{3}*{RATMSOR}&IT &38	&29	&28	&39	&50\\
		& &CPU &0.0139	&0.0109	&0.0105	&0.0138	&0.0171\\
		& &RES &9.37e-06	&6.73e-06	&8.89e-06	&8.69e-06	&9.70e-06\\
		\noalign{\smallskip}\hline\noalign{\smallskip}
		\multirow{9}*{$m=150$}
		&\multirow{3}*{GMSOR}&IT &44	&22	&130 &- &-\\
		& &CPU &0.1487	&0.0847	&0.3700 &- &- \\
		& &RES &7.40e-06	&5.09e-06	&9.20e-06 &- &- \\
		&\multirow{3}*{GTMSOR}&IT &61	&40	&82	&187 &-\\
		& &CPU &0.2037	&0.1155	&0.2405	&0.4909	&-\\
		& &RES &9.10e-06	&7.80e-06	&8.37e-06	&9.71e-06 &-\\
		&\multirow{3}*{RATMSOR}&IT &39	&29	&28	&39	&50\\
		& &CPU &0.1292	&0.1019	&0.0678	&0.0883	&0.1122\\
		& &RES &7.61e-06	&7.76e-06	&9.59e-06	&8.69e-06	&9.70e-06\\
		\noalign{\smallskip}\hline\noalign{\smallskip}
	\end{tabular}
\end{table}
\begin{exam}[\cite{Bai2010}]\label{example1}
	Let $m$ be a prescribed positive integer and $n = m^2$. Consider the LCP$(q, A)$, in which $A\in \mathbb{R}^{n\times n}$ is given by $A = \hat{A}+ \mu I$ and $q\in \mathbb{R}^n$ is given by $q = -Az^*$, where
	
	$$
	\begin{gathered}
	\hat{A} = tridiag(-I, S, -I)=
	\begin{pmatrix}
	S & -I & 0 & \cdots & 0 & 0 \\
	-I & S & -I & \cdots & 0 & 0 \\
	0 & -I & S & \cdots & 0 & 0 \\
	\vdots & \vdots & \vdots & \ddots & \vdots & \vdots \\
	0 & 0 & \cdots & \cdots & S & -I \\
	0 & 0 & \cdots & \cdots & -I & S \\
	\end{pmatrix}\in \mathbb{R}^{n\times n}\\
	\end{gathered}
	$$ is a block-tridiagonal matrix,
	
	$$
	\begin{gathered}
	S = tridiag(-1, 4, -1)=
	\begin{pmatrix}
	4 & -1 & 0 & \cdots & 0 & 0 \\
	-1 & 4 & -1 & \cdots & 0 & 0 \\
	0 & -1 & 4 & \cdots & 0 & 0 \\
	\vdots & \vdots & \vdots & \ddots & \vdots & \vdots \\
	0 & 0 & \cdots & \cdots & 4 & -1 \\
	0 & 0 & \cdots & \cdots & -1 & 4 \\
	\end{pmatrix}\in \mathbb{R}^{m\times m}\\
	\end{gathered}
	$$ is a tridiagonal matrix, and 
	\begin{equation*}
	z^*=(1, 2, 1, 2, \dots, 1, 2, \dots)^{T}\in \mathbb{R}^{n}\\
	\end{equation*}
	is the unique solution of the LCP$(q, A)$, see \cite{Bai2010} for more details.
	
	In \autoref{tab:1}-\ref{tab:3}, the iteration steps, the CPU times, and the residual norms of GMSOR, GTMSOR, ATMSOR and RATMSOR methods for the symmetric case are partially listed.
	
	In \autoref{tab:1}, the numerical results for RATMSOR with $m=200, \mu=2, \alpha=1$ are given when choosing different diagonal parameter matrix $\Omega_3$ and relaxation parameter $\theta$. Note that,  RATMSOR reduces to ATMSOR when $\Omega_3=0$. It is not difficult to see that the fewer ‘IT’ and less ‘CPU’ are obtained when choosing appropriate $\Omega_3$ and $\theta$. Particularly, the optimal parameters are obtained with $\Omega_3=\frac{1}{2}D_A, \theta=1.7$.
	
	In \autoref{tab:2}, the comparison results of GMSOR, GTMSOR, ATMSOR and RATMSOR with $\alpha=1, \Omega_3=\frac{1}{2}D_A, \theta=1.7$ are listed. It can be seen from the table that ‘IT’ and ‘CPU’ of the four methods increase with the increasing of $m$ when $\mu$ is fixed. Inversely, it is not hard to see that the ‘IT’ and ‘CPU’ of the four methods decrease with the increasing of $\mu$ when $m$ is fixed. Note that, by comparing GMSOR and ATMSOR, it is effective that gives two splittings of $A\Omega_1$. The table also shows that RATMSOR can be superior to GMSOR, GTMSOR and ATMSOR in computing efficiency by choosing parameters properly.
	
	In \autoref{tab:3}, we give some comparison results of GMSOR, GTMSOR, and RATMSOR with $\Omega_3=\frac{1}{2}D_A, \mu=2, \theta=1.7$ when $\alpha\not= 1$. From the table, it is easy to find that both GMSOR and RATMSOR have less ‘IT’ and ‘CPU’ than GTMSOR with $\alpha<1$. Nevertheless, both RATMSOR and GTMSOR have less ‘IT’ and ‘CPU’ than GMSOR with $\alpha>1$. Besides, it can be seen that RATMSOR has better computing efficiency in either case. Hence, the convergence interval of RATMSOR with respect to $\alpha$ is wider than that of GMSOR and GTMSOR.
	
\end{exam}
\begin{exam}[\cite{Bai2010}]\label{example2}
	Let $m$ be a prescribed positive integer and $n = m^2$. Consider the LCP$(q, A)$, in which $A\in \mathbb{R}^{n\times n}$ is given by $A = \hat{A}+ \mu I$ and $q\in \mathbb{R}^n$ is given by $q = -Az^*$, where
	
	$$
	\begin{gathered}
	\hat{A} = tridiag(-1.5I, S, -0.5I)=
	\begin{pmatrix}
	S & -0.5I & 0 & \cdots & 0 & 0 \\
	-1.5I & S & -0.5I & \cdots & 0 & 0 \\
	0 & -1.5I & S & \cdots & 0 & 0 \\
	\vdots & \vdots & \vdots & \ddots & \vdots & \vdots \\
	0 & 0 & \cdots & \cdots & S & -0.5I \\
	0 & 0 & \cdots & \cdots & -1.5I & S \\
	\end{pmatrix}\in \mathbb{R}^{n\times n}\\
	\end{gathered}
	$$ is a block-tridiagonal matrix,
	
	$$
	\begin{gathered}
	S = tridiag(-1.5, 4, -0.5)=
	\begin{pmatrix}
	4 & -0.5 & 0 & \cdots & 0 & 0 \\
	-1.5 & 4 & -0.5 & \cdots & 0 & 0 \\
	0 & -1.5 & 4 & \cdots & 0 & 0 \\
	\vdots & \vdots & \vdots & \ddots & \vdots & \vdots \\
	0 & 0 & \cdots & \cdots & 4 & -0.5 \\
	0 & 0 & \cdots & \cdots & -1.5 & 4 \\
	\end{pmatrix}\in \mathbb{R}^{m\times m}\\
	\end{gathered}
	$$ is a tridiagonal matrix, and 
	\begin{equation*}
	z^*=(1, 2, 1, 2, \dots, 1, 2, \dots)^{T}\in \mathbb{R}^{n}\\
	\end{equation*}
	is the unique solution of the LCP$(q, A)$.
	
	In \autoref{tab:4}-\ref{tab:6}, the iteration steps, the CPU times, and the residual norms of GMSOR, GTMSOR and RATMSOR methods for the nonsymmetric case are partially listed.
	
	In \autoref{tab:4}, the numerical results for RATMSOR with $m=200, \mu=2, \alpha=1$ are given when choosing different diagonal parameter matrix $\Omega_3$ and relaxation parameter $\theta$. It can be seen form the table that the fewer ‘IT’ and less ‘CPU’ are obtained when choosing $\Omega_3$ and $\theta$ properly. Specially, the optimal parameters are $\Omega_3=\frac{1}{2}D_A, \theta=1.7$.

	In \autoref{tab:5}, the comparison results of GMSOR, GTMSOR and RATMSOR with $\alpha=1, \Omega_3=\frac{1}{2}D_A, \theta=1.9$ are listed. It can be found that ‘IT’ and ‘CPU’ of the four methods increase with the increasing of $m$ when $\mu$ is fixed. Inversely, the ‘IT’ and ‘CPU’ of the four methods decrease with the increasing of $\mu$ when $m$ is fixed. In addition, from the table, it is shows that RATMSOR can be superior to GMSOR and GTMSOR in computing efficiency by choosing appropriate parameters.

	In \autoref{tab:6}, we list some comparison results of GMSOR, GTMSOR, and RATMSOR with $\Omega_3=\frac{1}{2}D_A, \mu=2, \theta=1.7$ when $\alpha\not= 1$. it is not hard to find that both GMSOR and RATMSOR have less ‘IT’ and ‘CPU’ than GTMSOR with $\alpha<1$. However, both RATMSOR and GTMSOR have less ‘IT’ and ‘CPU’ than GMSOR with $\alpha>1$. Moreover, in the table, we can find that RATMSOR has better computing efficiency in either case. Hence, the convergence interval of RATMSOR with respect to $\alpha$ is wider than that of GMSOR and GTMSOR.
\end{exam}
\section{Conclusions}
\label{sec:6}
In this paper, a relaxation accelerated two-sweep matrix splitting iteration method for solving LCP$(q, A)$ is established by utilizing the matrix splitting and introducing diagonal parameter matrix and relaxation parameter. Then, We give the convergence analysis of the RATMMS iteration method where the system matrix is an $H_+$-matrix. Numerical experiments have illustrated that the proposed method is more efficient than the existing numerous modulus-based methods by selecting appropriate parameters, moreover, choosing different parameters have significant influence on the computational efficiency of the proposed method. Hence, the choice of parameters is crucial. However, there are few theoretical researches on parameter selection. Therefore, the theoretical research on the selection of optimal parameters will be a worthy subject.
%



\bibliographystyle{elsarticle-num}
\bibliography{cas-refs}


\end{document}